\documentclass[10pt]{amsart}
\usepackage[english]{babel}
\usepackage{amssymb}
\usepackage{amsmath}
\usepackage{srcltx}
\usepackage{lscape}
\usepackage[usenames]{color}
\newtheorem{dummy}{Dummy}

\newtheorem{lemma}[dummy]{Lemma}
\newtheorem{theorem}[dummy]{Theorem}
\newtheorem{proposition}[dummy]{Proposition}
\newtheorem{corollary}[dummy]{Corollary}

\theoremstyle{definition}
\newtheorem{definition}{Definition}

\newtheorem{example}[dummy]{Example}

\newtheorem{remark}[dummy]{Remark}

\newcommand{\tl}{\widetilde{\tau}}

\newcommand{\ignore}[1]{}

\date{14.12.2015}
\author{S. Pumpl\"un}
\email{susanne.pumpluen@nottingham.ac.uk}
\address{School of Mathematical Sciences\\
University of Nottingham\\
University Park\\
Nottingham NG7 2RD\\
United Kingdom
}



\begin{document}

\title[Tensor products of nonassociative cyclic algebras]{Tensor products of nonassociative cyclic algebras}

\maketitle

\begin{abstract}
We study the tensor product of an associative and a nonassociative
 cyclic algebra. The condition for the tensor product
 to be a division algebra equals the classical one for the tensor product of
 two associative cyclic algebras by Albert or Jacobson, if the base field contains a suitable root of unity.
 Stronger conditions are obtained in special cases.
 Applications to space-time block coding are discussed.
\end{abstract}

%
%

\section*{Introduction}

Nonassociative cyclic algebras of degree $n$ are  canonical generalizations of associative cyclic algebras of degree $n$ and
 were first introduced over finite fields by Sandler \cite{S62}.
 Nonassociative quaternion algebras (the case $n=2$)
constituted  the first known examples of a nonassociative division algebra (Dickson \cite{D}).
Nonassociative cyclic algebras were investigated over arbitrary fields by
 Steele \cite{S12}, \cite{S13}, see also \cite{PS13.2}.

 In the following we study the tensor product
 $A=D_0\otimes_{F_0} D_1$ of an associative and a nonassociative
 cyclic algebra over a field $F_0$ and give conditions for $A$ to be a division algebra.
 We discover that   these tensor product algebras are
 used in space-time block coding \cite{P13}, \cite{PS13}, \cite{PS14},
and employed to construct some of the iterated codes by Markin and Oggier \cite{MO13}.

After recalling results needed in the paper in Section 1,
 results by Petit \cite{P66} are used to
show that the iterated algebras ${\rm It}_R^m(D,\tau,d)$ introduced in \cite{P13}, \cite{PS13}, \cite{PS14} can be defined using polynomials in skew-polynomial rings over $D$
when $D$ is a cyclic division algebra (Theorem \ref{thm:skew})  in Section 2.
In Section 3 we show that for
   an associative division algebra $D=(L/F_0,\sigma,c)\otimes_{F_0}F$, $L$ and $F$ linearly disjoint over $F_0$,
$$(L/F_0,\sigma,c)\otimes_{F_0} (F/F_0,\tau,d)\cong S_f\cong {\rm It}_R^m(D,\tau,d),$$
where the twisted polynomial $f(t)=t^m-d\in D[t;\widetilde{\tau}^{-1}]$, $\widetilde{\tau}$ an automorphism of
$D$ canonically extending $\tau$,
is used to construct the algebra $S_f$  (Theorem \ref{main}).
Section 3 contains the main results:
if $D_0$ is an associative cyclic algebra over $F_0$ such that $D=D_0\otimes_{F_0}F$
is a division algebra, and $D_1=(F/F_0,\tau,d)$  a nonassociative cyclic algebra of degree $m$,
$L$ and $F$ linearly disjoint over $F_0$, then
$D_0\otimes_{F_0} D_1$
is a division algebra if and only if
$f(t)=t^m-d$ is irreducible in $D[t;\widetilde{\tau}^{-1}]$
 (Theorem \ref{thm:classicalnew}).

  Assume further that $m$ is prime and if $m\not=2,3$, that $F_0$ contains a primitive $m$th root of unity. Then
$(L/F_0,\sigma,c)\otimes_{F_0} (F/F_0,\tau,d)$ is a division algebra if and only if
$$d\not=z\widetilde{\tau}(z)\cdots \widetilde{\tau}^{m-1}(z)$$ for all $z\in D$.
 This generalizes the classical condition
 for the tensor product of two associative cyclic algebras \cite[Theorem 1.9.8]{J96},
 (see also Theorem \ref{associativetensor}), to the nonassociative setting.
 Some more detailed conditions are obtained for special cases.
For instance, if ${\rm char}\,F_0\not=2$, $D_0$ is an associative quaternion algebra
over $F_0$ which remains a division
algebra over $F$,
 and $D_1$ a nonassociative quaternion algebra, such that $D_0$ and $D_1$ do not share a common subfield,
 then $D_0\otimes_{F_0} D_1$
 is always a division algebra over $F_0$ (Theorem \ref{biquat}).
In Section 4, we discuss how our results can be applied to systematically construct fully diverse
fast-decodable space-time block codes for $nm$ transmit and $m$ receive antennas.
We thus generalize the set-up discussed in \cite[Section V.A]{MO13} which is limited to tensor products
of a cyclic algebra of degree three and a nonassociative
quaternion algebra, i.e. to $n=3$, $m=2$, and see that Theorems \ref{thm:classicalnew},
\ref{biquat} and \ref{3xquat} provide  conditions for an iterated code
consisting of larger matrices (where $m>2$) to be fully diverse.
We then design a new family of fast-decodable fully diverse $4\times 2$-codes
with  non-vanishing determinant starting with the Silver code,
which have decoding complexity $O(M^{4.5})$, i.e. the  same decoding complexity  as all state-of the art
fast-decodable fully diverse $4\times 2$ codes.
Thus our construction allows us to systematically design codes whose decoding complexity
is competitive with the
ones designed `ad hoc' like the SR-code \cite[IV.B]{R13} or the code in \cite{LHHC}
 and which also have non-vanishing determinant.

\section{Preliminaries}

\subsection{Nonassociative algebras}

Let $F$ be a field and let $A$ be an $F$-vector space.
We call $A$ an \emph{algebra} over $F$ if there exists an
$F$-bilinear map $A\times A\to A$, $(x,y) \mapsto x \cdot y$, denoted simply by juxtaposition $xy$,
the  \emph{multiplication} of $A$.
An algebra $A$ is called \emph{unital} if there is
an element in $A$, denoted by 1, such that $1x=x1=x$ for all $x\in A$.
We will only consider unital algebras.

   An algebra $A\not=0$ is called a {\it division algebra} if for any $a\in A$, $a\not=0$,
the left multiplication  with $a$, $L_a(x)=ax$,  and the right multiplication with $a$, $R_a(x)=xa$, are bijective.
If $A$ is finite-dimensional, $A$ is a division algebra if and only if $A$ has no zero divisors \cite[pp. 15, 16]{Sch}.

For an $F$-algebra $A$, associativity in $A$ is measured by the {\it associator} $[x, y, z] = (xy) z - x (yz)$.
The {\it left nucleus} of $A$ is defined as ${\rm Nuc}_l(A) = \{ x \in A \, \vert \, [x, A, A]  = 0 \}$, the
{\it middle nucleus}  is
 ${\rm Nuc}_m(A) = \{ x \in A \, \vert \, [A, x, A]  = 0 \}$,  the
{\it right nucleus}  is  ${\rm Nuc}_r(A) = \{ x \in A \, \vert \, [A,A, x]  = 0 \}$ and the
{\it  nucleus}  is
 ${\rm Nuc}(A) = \{ x \in A \, \vert \, [x, A, A] = [A, x, A] = [A,A, x] = 0 \}$.
 It is an associative
subalgebra of $A$ containing $F1$ and $x(yz) = (xy) z$ whenever one of the elements $x, y, z$ is in
${\rm Nuc}(A)$. The {\it commuter} of $A$ is defined as ${\rm Comm}(A)=\{x\in A\,|\,xy=yx \text{ for all }y\in A\}$
and the {\it center} of $A$ is ${\rm C}(A)=\{x\in A\,|\, x\in \text{Nuc}(A) \text{ and }xy=yx \text{ for all }y\in A\}$.
For two nonassociative algebras  $C$ and $D$ over $F$,
$${\rm Nuc}(C)\otimes_F {\rm Nuc}(D)\subset {\rm Nuc}(C \otimes_F D).$$
Thus we can consider the tensor product $A=C \otimes_F D$ as a right $R$-module over
any ring $R\subset {\rm Nuc}(C)\otimes_F {\rm Nuc}(D)$.

\subsection{Associative and nonassociative cyclic algebras}

Let $K/F$ be a cyclic Galois extension of degree $n$ with Galois group ${\rm Gal}(K/F)=\langle \sigma \rangle$.

An associative cyclic algebra $(K/F,\sigma,c)$ of degree $n$ over $F$, $c\in F^\times$, is an $n$-dimensional $K$-vector space
\[
(K/F,\sigma,c)=K \oplus eK \oplus e^2 K\oplus\dots \oplus e^{n-1}K,
\]
with multiplication given by the relations
$$\label{eq:rule}
e^n=c,~le=e\sigma(l),
$$
for all $l\in K$.  $(K/F,\sigma,c)$ is division
 for all $c\in F^\times$, such that
  $c^s\not\in N_{K/F}(K^\times)$ for all $s$ which are prime divisors of $n$, $1\leq s\leq n-1$.

For $c\in K\backslash F$, we define a unital nonassociative algebra  $(K/F,\sigma,c)$ (Sandler \cite{S62})
 as the $n$-dimensional $K$-vector space
\[
(K/F,\sigma,c)=K \oplus eK \oplus e^2K \oplus \dots\oplus e^{n-1}K,
\]
 where multiplication is given by the following rules
for all $a,b\in K, 0 \leq i,j, <n$, which then are extended linearly to all elements of $A$:
\[
 (e^ia)(e^jb) =
  \begin{cases}
  e^{i+j}  \sigma^j(a)b  & \text{if } i+j < n, \\
   e^{(i+j)-n}  d \sigma^j(a)b & \text{if } i+j \geq n.
  \end{cases}
\]
We call $D=(K/F,\sigma,c)$ with $c\in K\setminus F$  a  \emph{nonassociative cyclic algebra of degree} $n$.
$D$ has nucleus  $K$ and center  $F$.
$D$ is not $(n+1)$th power associative since $(e^{n-1}e)e = e\sigma(a)$ and $e(e^{n-1}e) = ea $.

 If $[K:F]$ is prime, $D$ always
is a division algebra. If $[K:F]$ is not prime, $D$ is a division algebra for any choice of $c$ such that
$1,c,\dots,c^{n-1}$ are linearly independent over $F$ \cite{S13}.

For $n=2$, $(K/F,\sigma,c)={\rm Cay}(K,c)$ is an associative (if $c\in F$) or nonassociative
(if $c\in K\setminus F$)
quaternion algebra over $F$, cf. {\cite{As-Pu}}, \cite{PU11} or {\cite{W}}.

From now on, when we say $D=(K/F, \sigma, c)$ is a cyclic algebra, we mean an associative or
nonassociative cyclic algebra over $F$ without always explicitly stating that we also allow $c\in K^\times$.
We call $\{1,e,e^2,\dots,e^{n-1}\}$ the {\it standard basis} of $(K/F,\sigma,c)$.

$D=(K/F,\sigma,c)$
 is a $K$-vector space of dimension $n$  (since $K= {\rm Nuc}(D)$ if the algebra is nonassociative) and,
after a choice of a $K$-basis, we can embed the $K$-vector space ${\rm End}_K(D)$ into
${\rm Mat}_n(K)$.
The left multiplication of elements of $D$ with $y=y_0+ey_1+\dots+e^{n-1}y_{n-1}\in D$
($y_i\in K$) induces the
$K$-linear embedding $\lambda:D\rightarrow {\rm Mat}_{n}(K)$.

%
%

\subsection{Iterated algebras}

From now on, we will use the following notation:
Let $F$ and $L$ be fields and let $K$ be a cyclic field extension of both $F$ and $L$ such that
\begin{enumerate}
\item $Gal(K/F) = \langle \sigma \rangle$ and $[K:F] = n$,
\item $Gal(K/L) = \langle \tau \rangle$ and $[K:L] = m$,
\item $\sigma$ and $\tau$ commute: $\sigma \tau = \tau \sigma$.
\end{enumerate}
Define $F_0=F\cap L$.
Let $D=(K/F, \sigma, c)$ be a cyclic  algebra of degree $n$ over $F$ with reduced norm $N_{D/F}$.
 For $x= x_0 + ex_1 + e^2x_2 +\dots + e^{n-1}x_{n-1}\in D$ ($x_i\in K$, $1\leq i\leq n$), define the $L$-linear map
 $\widetilde{\tau}:D\to D$  via
 $$\widetilde{\tau}(x)=\tau(x_0) + e \tau(x_1) + e^2\tau(x_2) +\dots + e^{n-1}\tau(x_{n-1}).$$
 If $c\in F_0$ then
 $$ \tl(xy) = \tl(x)\tl(y) \text{ and }\lambda (\tl(x)) = \tau (\lambda(x))$$
 for all $x,y \in D$,
where for any matrix $X=\lambda(x)$ representing left multiplication with $x$, $\tau(X)$
means applying $\tau$ to each entry of the matrix.

\begin{definition} (cf. \cite{PS13}, \cite{PS14})
Pick $d \in F^\times$, $c\in F_0$. Define a multiplication on the right $D$-module
\[ D \oplus fD \oplus f^2D \oplus \cdots \oplus f^{m-1}D,\]
via the rules
\[
 (f^i x)(f^j y) =
  \begin{cases}
   f^{i+j} \tl^j(x)y & \text{if } i+j < m \\
   f^{(i+j)-m} \tl^j(x)yd  & \text{if } i+j \geq m
  \end{cases}
\]
for all $x, y \in D$, and call the resulting  {\it iterated algebra} ${\rm It}_R^m(D, \tau, d) $.
\end{definition}

 ${\rm It}_R^m(D,\tau,d)$ is a nonassociative algebra over  $F_0$ of dimension $m^2n^2$
 with unit element $(1_D,0,\dots,0)$ and contains $D$ as a subalgebra.
 We call
$$\{1,e,e^2,\dots,e^{n-1},f,fe,fe^2,\dots,f^{m-1}e^{n-1}\}$$
 the {\it standard basis} of the  $K$-vector space ${\rm It}_R^m(D,\tau,d)$.

\begin{example}
Let $A={\rm It}_R^3(D,\tau,d)$ and put
 $f=(0,1_D,0)$. Then $f^2=(0,0,1_D)$, $f^2f=(d,0,0)=ff^2$ and the multiplication in $A$ is given by
\[(u,v,w)(u',v',w')
=(
\begin{bmatrix}
u & d\widetilde{\tau}(w) & d\widetilde{\tau}^2(v)  \\
v & \widetilde{\tau}(u) & d\widetilde{\tau}^2(w)  \\
w & \widetilde{\tau}(v) & \widetilde{\tau}^2(u)  \\
\end{bmatrix}
\left [\begin {array}{c}
u'  \\
v'  \\
w'
\end {array}\right ])^T\]
$$
=(uu'+d\widetilde{\tau}(w)v'+d\widetilde{\tau}^2(v)w',
                     vu'+\widetilde{\tau}(u)v'+d\widetilde{\tau}^2(w)w',
                     wu'+\widetilde{\tau}(v)v'+\widetilde{\tau}^2(u)w')$$
for $u,v,w,u',v',w'\in D$.

\end{example}

From now on, let $$A={\rm It}_R^m(D,\tau,d).$$

\begin{lemma}\label{lem:lem5}
 (i) The cyclic algebra $(K/L,\tau,d)$ over $L$, viewed as an algebra over $F_0$,
 is a subalgebra of $A$, and is nonassociative if $d\in F\setminus F_0$.
 \\ (ii) Let $m$ be even.
Then  $It_R^2(D, \tau, d)$ is isomorphic to a subalgebra of $A$.

 In particular, the quaternion algebra
$(K/L,\tau,d)={\rm Cay}(K,d)$ over $L$, viewed as algebra over $F_0$, is a subalgebra of ${\rm It}_R^2(D,\tau,d)$,
which is nonassociative and division if $d\in F\setminus F_0$.
\end{lemma}

\begin{proof}
(i) This can be seen by restricting the multiplication of $A$ to $K\oplus \dots\oplus K$.
\\ (ii)
 Suppose that $m = 2s$ for some integer $s$. Then  ${\rm It}_R^2(D, \tau, d)$
is isomorphic to $D \oplus f^sD$, which is a subalgebra of $A$
 under the multiplication inherited from $A$.
\end{proof}

We can embed ${\rm End}_K(A)$ into the module
${\rm Mat}_{nm}(K)$.
Left multiplication
$L_x$ with $x\in A$  is a right $K$-endomorphism, so that we obtain  a well-defined additive map
$$L:A\to {\rm End}_K(A)\hookrightarrow {\rm Mat}_{nm}(K),\quad x\mapsto L_x \mapsto L(x)=X$$
which is injective if $A$ is division.

Take the standard basis
$\{1, e, \ldots, e^{n-1}, f, fe, \ldots, f^{m-1}e^{n-1}\}$
 of the  $K$-vector space $A$.
Then
\[xy = (\lambda(M(x)) y^T)^T,\]
where
\begin{equation} \label{equ:main}
\lambda(M(x)) = \left[ \begin{array}{cccc}
\lambda(x_0) & d\tau(\lambda(x_{m-1}))&  \cdots & d \tau^{m-1}(\lambda(x_1)) \\
\lambda(x_1) & \tau(\lambda(x_0)) & \cdots & d \tau^{m-1}(\lambda(x_{2})) \\
\vdots & \vdots  & \ddots & \vdots \\
\lambda(x_{m-1}) & \tau(\lambda(x_{m-2})) & \cdots & \tau^{m-1}(\lambda(x_0)) \end{array} \right]
\end{equation}
 is obtained by taking the matrix $\lambda(x_i)$, $x_i \in D$, representing left multiplication in $D$ of each entry in the
 matrix $M(x)$.

$\lambda(M(x))$ represents the left multiplication by the element $x$ in $A$. Define
 $$M_{A}:A\to K, \quad M_{A}(x)=\det(\lambda(M(x))).$$
 If $x \in A$ is nonzero and not a left zero divisor in $A$, then $M_A(x)\not=0$ by \cite[Theorem 9]{PS14}.

\begin{remark}  It is clear that $A$ is a division algebra if and only if $M_A(x)\not=0$ for all $x\not=0$:
 If $A$ is a division algebra then $L_x$ is bijective for all $x\not=0$ and thus $\lambda(M(x))$ invertible, i.e.
$M_A(x)\not=0$. Conversely, if $M_A(x)\not=0$ for all $x\not=0$ then for all $x,y\in A$, $x\not=0$,
$y\not=0$, also $xy=(\lambda(M(x)) y^T)^T\not=0$.

\end{remark}

\section{Division algebras obtained from skew-polynomial rings}

In the following, we recall results from \cite{J96} and \cite{P66}.
Let $D$ be a unital division ring and $\sigma$ a ring isomorphism of $D$. The \emph{twisted polynomial ring} $D[t;\sigma]$
is the set of polynomials
$$a_0+a_1t+\dots +a_nt^n$$
with $a_i\in D$, where addition is defined term-wise and multiplication by
$$ta=\sigma(a)t \quad (a\in D).$$
That means,
$$at^nbt^m= a \sigma^{n}(b)t^{n+m} \text{ and } t^na=\sigma^n(a)t^n$$
for all $a,b\in D$
\cite[p.~2]{J96}.
$R=D[t;\sigma]$ is a left principal ideal domain  and
there is a right division algorithm in $R$ \cite[p.~3]{J96}, i.e. for all $g,f\in R$, $g\not=0$, there exist unique $r,q\in R$
  such that ${\rm deg}(r)<{\rm deg}(f)$ and
$$g=qf+r.$$
 $R=D[t;\sigma]$ is also a right principal ideal domain  \cite[p.~6]{J96}
with a left division algorithm in $R$ \cite[p.~3 and Prop. 1.1.14]{J96}.
(We point out that our terminology is the one used by Petit \cite{P66} and Lavrauw and Sheekey \cite{LS};
it is different from Jacobson's \cite{J96}, who calls what we call right a left division algorithm and vice versa.)

Thus $R=D[t;\sigma]$ is a (left and right) principal ideal domain (PID).

 An element $f\in R$ is \emph{irreducible} in $R$ if  it is no unit and it has no proper factors, i.e there do not exist $g,h\in R$ with
 ${\rm deg}(g),{\rm deg} (h)<{\rm deg}(f)$ such
 that $f=gh$ \cite[p.~11]{J96}.

 \begin{definition} (cf. \cite[(7)]{P66})
 Let $f\in D[t;\sigma]$ be of degree $m$ and let ${\rm mod}_r f$ denote the remainder of right division by $f$.
 Then the vector space $R_m=\{g\in D[t;\sigma]\,|\, {\rm deg}(g)<m\}$ together with the multiplication
 $$g\circ h=gh \,\,{\rm mod}_r f $$
 becomes a unital nonassociative algebra $S_f=(R_m,\circ)$ over $F_0=\{z\in D\,|\, zh=hz \text{ for all } h\in S_f\}$.
\end{definition}

The multiplication is well-defined because of the right division algorithm
and $F_0$  is a subfield of $D$.

 Since $\sigma$ is a ring isomorphism, we can use the left division algorithm to define an algebra as well:
 Let $f\in D[t;\sigma]$ be of degree $m$ and let ${\rm mod}_l f$ denote the remainder of left division by $f$.
 Then  $R_m$ together with the multiplication
 $$g\circ h=gh \,\,{\rm mod}_l f $$
 becomes a nonassociative algebra $\,_fS=(R_m,\circ)$, which, however, turns out to be anti-isomorphic to a suitable
 algebra $S_g$ for some $g\in R'$ in some twisted polynomial ring $R'$.

\begin{remark}
(i) When ${\rm deg}(g){\rm deg}(h)<m$, the multiplication of $g$ and $h$ in $S_f$ is the same as the multiplication
of $g$ and $h$ in $R$ \cite[(10)]{P66}.
\\
 (ii)
Given a cyclic Galois field extension $K/F$ of degree $m$
with ${\rm Gal}(K/F)=\langle \sigma\rangle$, the cyclic algebra $(K/F,\sigma,d)$ is the algebra
$S_f$ with $f(t)=t^m-d\in R=K[t;\sigma^{-1}]$ \cite[p.~13-13]{P66}.
\\  (iii) Let $D$ be a finite-dimensional central division algebra over $F$ and $\sigma$ an automorphism
of $D$ of order $m$.
In \cite{J96},  the associative algebras
 $$E(f)=\{ g\in D[t;\sigma]\,|\, {\rm deg}(g)<m,\, f\text{ right divides }fg \}$$
for $f=t^m-d\in D[t;\sigma]$, were investigated. $E(f)$ is division iff $f$ is irreducible.
\end{remark}

\begin{theorem} \label{thm:petit}
(cf. \cite[(2), p.~13-03, (7) (9), (15), (17), (18), (19)]{P66})
Let $f=t^m-\sum_{i=0}^{m-1}d_it^i\in R=D[t;\sigma]$.
\\
(i) If $S_f$ is not associative then
$${\rm Nuc}_l(S_f)={\rm Nuc}_m(S_f)=D$$ and
$${\rm Nuc}_r(S_f)=\{g\in S_f\,|\, fg\in Rf\}=E(f).$$
\\ (ii) If $f$ is irreducible then ${\rm Nuc}_r(S_f)$ is an associative division algebra.
\\ (iii) Let $f\in R$ be irreducible and $S_f$ a finite-dimensional $F_0$-vector space
or a finite-dimensional right ${\rm Nuc}_r(S_f)$-module. Then $S_f$ is a division algebra.
\\ (iv) $f(t)=t^2-d_1t-d_0$ is irreducible in $D[t;\sigma]$ if and only if $\sigma(z)z-d_1z-d_0\not=0$ for all $z\in D$.
\\ (v) $f(t)=t^3-d_2t^2-d_1t-d_0$ is irreducible in $D[t;\sigma]$ if and only if
$$\sigma^2(z)\sigma(z)z-\sigma^2(z)\sigma(z)d_2-\sigma^2(z)\sigma(d_1)-\sigma^2(d_0)\not=0$$
and
$$\sigma^2(z)\sigma(z)z-d_2\sigma(z)z-d_1z-d_0\not=0$$
 for all $z\in D$.
 \\ (vi) Suppose $m$ is prime and ${\rm C}(D)\cap {\rm Fix}(\sigma)$ contains a primitive $m$th root of unity.
Then $f(t)=t^m-d$ is irreducible in $D[t;\sigma]$ if and only if
$$d\not=\sigma^{m-1}(z)\cdots\sigma(z)z
\text{ and }\sigma^{m-1}(d)\not=\sigma^{m-1}(z)\cdots\sigma(z)z$$
 for all $z\in D$.
\end{theorem}

Theorem \ref{thm:petit}, parts (v) and (vi), can be improved as follows
\cite{B}:

\begin{theorem}\label{thm:betterpetit}
(i)  $f(t)=t^3-d$ is irreducible in $D[t;\sigma]$ if and only if
$$d\not=\sigma^2(z)\sigma(z)z$$
for all $z\in D$.
 \\ (ii)  $f(t)=t^4-d$ is irreducible in $D[t;\sigma]$ if and only if
$$d\not=\sigma^3(z)\sigma^2(z)\sigma(z)z  $$
and
$$ \sigma^2(z_1)\sigma(z_1)z_1+\sigma^2(z_0)z_1+\sigma^2(z_1)
\sigma(z_0)\not=0 \text{ or } \sigma^2(z_0)z_0+\sigma^2(z_1)\sigma(z_0)z_0\not=d$$
 for all $z_0,z_1\in D$.
 \\ (iii) Suppose $m$ is prime and ${\rm C}(D)\cap {\rm Fix}(\sigma)$ contains a primitive $m$th root of unity.
Then $f(t)=t^m-d\in D[t;\sigma]$ is irreducible if and only if
$$d\not=\sigma^{m-1}(z)\cdots\sigma(z)z$$
 for all $z\in D$.
\end{theorem}

\begin{proof} (i)
 If $f(t)=t^3-d$ is reducible then either $f$ is divisible by a linear factor from the left or from the right.
A straightforward calculation  shows that  $f$ is divisible on the left by
 $t-z$, $z\in D$, iff $\sigma^2(d)=\sigma^2(z)\sigma(z)z$ iff $z\sigma^{-1}(z)\sigma^{-2}(z)=d$.
By \cite[1.3.11]{J96}, $f$ is divisible on the right by $t-z$, $z\in D$, iff $0=\sigma^2(z)\sigma(z)z-d$,
 which is the remainder of right division of $f$ by $t-z$.
 If $f(t)=t^3-d$ is irreducible then $t-z$ does not divide $f(t)$ from the right for any $z\in D$ and thus
$d\not=\sigma^2(z)\sigma(z)z$
for all $z\in D$.

Conversely, assuming $f(t)$ is reducible then there is a linear factor dividing $f$ from the left or from
the right. If $f(t)=g(t)(t-b)$ then $d=\sigma^2(b)\sigma(b)b$. If $f(t)=(t-c)g(t)$ this is equivalent to
 $d=c\sigma^{-1}(c)\sigma^{-2}(c)$. This implies that $(t^2+ct+c\sigma^{-1})(t-\sigma^{-2}(c))=t^3-c\sigma^{-1}(c)\sigma^{-2}(c)
 =d$ and that $t-\sigma^{-2}(c)$ divides $f(t)$ from the right. Put $b=\sigma^{-2}(c)$ to see there is $b\in D$ such that
 $\sigma^2(b)\sigma(b)b=d$.
\\ (ii)
 If $f(t)=t^4-d$ is reducible then either $f$ is divisible by a linear factor from the left, from the right, or
$f=g_1(t)g_2(t)$ for two irreducible polynomials $g_1,g_2\in R$ of degree 2.
By \cite[1.3.11]{J96},
 $f$ is divisible on the right by
a factor $t-z$, $z\in D$, iff $d=\sigma^3(z)\sigma^2(z)\sigma(z)z$.
A straightforward calculation  shows that $f$ is divisible on the left by
a factor $t-z$, $z\in D$, iff $0=\sigma^3(z)\sigma^2(z)\sigma(z)z-\sigma^3(d)$,
 which is the remainder of left division of $f$ by $t-z$.
Moreover, $f$ is divisible  on the right by $g(t)=t^2-z_1t-z_0\in R$ iff
$$[\sigma^2(z_1)\sigma(z_1)z_1+\sigma^2(z_0)z_1+\sigma^2(z_1)
\sigma(z_0)]t+[\sigma^2(z_0)z_0+\sigma^2(z_1)\sigma(z_0)z_0-d]=0$$
 which is the remainder of right division of $f(t)$ by $g(t)$. This is equivalent to
 $$\sigma^2(z_1)\sigma(z_1)z_1+\sigma^2(z_0)z_1+\sigma^2(z_1)
\sigma(z_0)=0\text{ and } \sigma^2(z_0)z_0+\sigma^2(z_1)\sigma(z_0)z_0=d.$$
Thus $f$ is irreducible if and only if
$$d\not=z\sigma(z)\sigma^2(z)\sigma(z)\text{ and }\sigma^3(d)\not=z\sigma(z)\sigma^2(z)\sigma^3(z)$$
and
$$ \sigma^2(z_1)\sigma(z_1)z_1+\sigma^2(z_0)z_1+\sigma^2(z_1)
\sigma(z_0)\not=0 \text{ or } \sigma^2(z_0)z_0+\sigma^2(z_1)\sigma(z_0)z_0\not=d$$
 for all $z_0,z_1\in D$.
Now observe that the case that $f(t)$ is divisible on the left by some $t-z$ is equivalent to
$d=z\sigma^{-1}(z)\sigma^{-2}(z)\sigma^{-3}(z)$ and putting $w=\sigma^{-3}(z)$,
we obtain
$$d=z\sigma^{-1}(z)\sigma^{-2}(z)\sigma^{-3}(z)=\sigma^{3}(w)\sigma^{2}(w)\sigma(w)w ,$$
which is equivalent to $f$ being divisible on the right by $t-\sigma^{-3}(z)$,
and we have proved the assertion.
 \\ (iii) By \cite[Ex no 4, p.~344]{Bou}, $f$ is either irreducible or a product of $m$ linear factors.
Thus $f$ is irreducible if and only if for all $z\in D$, $t-z$ does not divide $f$ on the right, which is equivalent to
the assertion.
\end{proof}

The iterated algebras ${\rm It}_R^m(D,\tau,d)$ were originally
introduced for space-time coding (\cite{MO13}, \cite{R13}, \cite{PS14}), and can be obtained from skew-polynomial rings:

\begin{theorem} \label{thm:skew}
Let $F$ and $L$ be fields, $F_0=F\cap L$, and let $K$ be a cyclic field extension of both $F$ and $L$ such that
\begin{enumerate}
\item $Gal(K/F) = \langle \sigma \rangle$ and $[K:F] = n$,
\item $Gal(K/L) = \langle \tau \rangle$ and $[K:L] = m$,
\item $\sigma$ and $\tau$ commute: $\sigma \tau = \tau \sigma$.
\end{enumerate}
 Let $D=(K/F, \sigma, c)$  be an associative cyclic division algebra over $F$ of degree $n$,
  $c\in F_0$ and $d \in D^\times$.
Then
$${\rm It}_R^m(D,\tau,d)= S_f$$
 where  $R=D[t;\widetilde{\tau}^{-1}]$ and $f(t)=t^m-d$.
\end{theorem}

\begin{proof}
Let $f=(0,1_D,0,\dots,0)\in A={\rm It}_R^m(D,\tau,d)$.
The multiplication on
$A= D \oplus fD \oplus f^2D \oplus \cdots \oplus f^{m-1}D$
is given by
\[
 (f^i x )(f^j y) =
  \begin{cases}
   f^{i+j}  \tl^j(x)y  & \text{if } i+j < m \\
    f^{(i+j)-m} \tl^j(x) yd  & \text{if } i+j \geq m
  \end{cases}
\]
for all $x, y \in D$  which corresponds to the multiplication of the algebra $S_f$.
\end{proof}

Theorems \ref{thm:petit}, \ref{thm:betterpetit} and \ref{thm:skew} imply:

\begin{corollary} \label{cor:nec}
Assume the setup of Theorem \ref{thm:skew}.
\\ (i) If $d\not\in F_0$ then
$${\rm Nuc}_l({\rm It}_R^m(D,\tau,d))={\rm Nuc}_m({\rm It}_R^m(D,\tau,d))=D$$ and
$${\rm Nuc}_r({\rm It}_R^m(D,\tau,d))=\{g\in S_f\,|\, fg\in Rf\}.$$
\\ (ii) ${\rm It}_R^m(D,\tau,d)$ is a division algebra if and only if $f(t)$ is irreducible in $D[t;\widetilde{\tau}^{-1}]$.
\\ (iii) ${\rm It}_R^4(D,\tau,d)$ is a division algebra if and only if
$$d\not=z\widetilde{\tau}(z)\widetilde{\tau}^2(z)\widetilde{\tau}^3(z)$$
and
$$ \widetilde{\tau}^2(z_1)\widetilde{\tau}^3(z_1)z_1+\widetilde{\tau}^2(z_0)z_1+\widetilde{\tau}^2(z_1)
\widetilde{\tau}^3(z_0)\not=0 \text{ or } \widetilde{\tau}^2(z_0)z_0+\widetilde{\tau}^2(z_1)\widetilde{\tau}^3(z_0)z_0\not=d$$
 for all $z,z_0,z_1\in D$.
\\ (iv) Suppose that $m$ is prime and in case $m\not=2,3$, additionally
 that $F_0$ contains a primitive
$m$th root of unity.
Then ${\rm It}_R^m(D,\tau,d)$ is a division algebra  if and only if
$$d\not=z\widetilde{\tau}(z)\cdots \widetilde{\tau}^{m-1}(z)$$ for all $z\in D$.
\end{corollary}

\begin{lemma} \label{le:nec}
Assume the setup of Theorem \ref{thm:skew} and $d\in F^\times$.
 If $\tau(d^n)\not=d^n$, then
 $d\not=z\widetilde{\tau}(z)\cdots \widetilde{\tau}^{m-1}(z)$
  for all $z\in D$.
\end{lemma}

The proof generalizes the idea of the proof of \cite[Proposition 13]{MO13}:

\begin{proof}
If $d=z\widetilde{\tau}(z)\cdots \widetilde{\tau}^{m-1}(z)$ for some $z\in D$, then for $Z=\lambda(z)$ this means
$$Z \tau(Z)\cdots \tau^{m-1}(Z)=d I_n$$ and therefore
$\det (Z) \det(\tau(Z))\cdots\det(\tau^{m-1}(Z)) =d^n$. Since the left-hand-side is fixed by $\tau^i$,  this
implies that $\tau^i(d^n)=d^n$ for $1\leq i <m$, in particular, $\tau(d^n)=d^n$.
\end{proof}

\begin{corollary} \label{cor:necII}
Assume the setup of Theorem \ref{thm:skew} and $d\in F^\times$.
 Suppose that $m$ is prime and in case $m\not=2,3$, additionally that $F_0$ contains a primitive $m$th root of unity.
\\ (i) If $\tau(d^n)\not=d^n$ then ${\rm It}_R^m(D,\tau,d)$ is a division algebra.
\\ (ii)
 If $d\in F$ such that $d^n\not\in N_{D/F_0}(D^\times)$, then ${\rm It}_R^m(D,\tau,d)$
 is a division algebra. In particular, for all $d\in F\setminus F_0$ with $d^n\not\in F_0$, ${\rm It}_R^m(D,\tau,d)$
 is a division algebra.
\end{corollary}

\begin{proof} (i) is clear.
\\ (ii)
Since $c\in F_0\subset {\rm Fix}(\tau)=L$ we have
$N_{D/F}(\widetilde{\tau}(x))=\tau(N_{D/F}(x))$ for all $x\in D$ by \cite[Proposition 4]{P13}.
Assume that $d=z\tl(z)\cdots\tl^{m-1}(z)$, then
$$N_{D/F}(d)=N_{D/F}(z)N_{D/F}(\tl(z))\cdots N_{D/F}(\tl^{m-1}(z))=N_{D/F}(z) \tau(N_{D/F}(z))\cdots \tau^{m-1}
(N_{D/F}(z)).$$
Put $a=N_{D/F}(z)$, note that $a \tau(a)\cdots\tau^{m-1}(a)=N_{F/F_0}(a)=N_{F/F_0}(N_{D/F}(z))=N_{D/F_0}(z)\in F_0$,
 and use that $N_{D/F}(d)=d^n$ for $d\in F$.
\end{proof}

\section{The tensor product of an associative and a nonassociative cyclic algebra}

\subsection{} Let $L/F_0$ be a cyclic Galois field extension of degree $n$ with ${\rm Gal}(L/F_0)=\langle \sigma\rangle$,
and $F/F_0$ be a cyclic Galois field extension of degree $m$ with ${\rm Gal}(F/F_0)=\langle \tau\rangle$.
Let $L$ and $F$ be linearly disjoint over $F_0$ and let
$K=L\otimes_{F_0} F=L\cdot F$ be the composite of $L$ and $F$ over $F_0$, with Galois group
${\rm Gal}(K/F_0)=\langle \sigma\rangle\times \langle \tau\rangle$, where $\sigma$ and
$\tau$ are canonically extended to $K$.

In the following, let $D_0=(L/F_0,\sigma,c)$ and $D_1=(F/F_0,\tau,d)$ be two  cyclic  algebras  over $F_0$
such that $D_0$ is associative and $D_1$ is nonassociative, i.e.
$c\in F_0^\times$ and $d\in F^\times$. Let
$$A=(L/F_0,\sigma,c)\otimes_{F_0} (F/F_0,\tau,d).$$
Then $K$  is a subfield of $A$ of degree $mn$ over $F_0$ and $K=L\otimes_{F_0} F\subset {\rm Nuc}(A)$.

Let $\{1,e,e^2,\dots,e^{n-1}\}$
be the standard basis of the  $L$-vector space $D_0$ and
$\{1,f,f^2,
\dots,
\\ f^{m-1}\}$
be the standard basis of the $F$-vector space $D_1$.
$A$ is a $K$-vector space with basis
$$\{1\otimes 1, e\otimes 1,\dots, e^{n-1}\otimes 1 ,1\otimes f, e\otimes f, \dots, e^{n-1}\otimes f^{m-1}\}.$$
 Identify
$$A=K\oplus eK \oplus \dots\oplus e^{n-1}K\oplus fK  \oplus efK \oplus \dots\oplus e^{n-1}f^{m-1}K.$$
Note that  $D_0\otimes_{F_0} F=(K/F,\sigma,c)$.
An element in $\lambda(A)$ has the form
\begin{equation} \label{eq:nbynass}
\left[ \begin{array}{ccccc}
Y_0     &  d\tau(Y_{n-1}) & d\tau^2(Y_{n-2}) & \dots  & d\tau^{m-1}(Y_1) \\
Y_1     & \tau(Y_0)            & d\tau^2(Y_{n-1}) & \dots  & d\tau^{m-1}(Y_2) \\
\vdots  &                        &          \vdots         &        &  \vdots                   \\
Y_{n-2} & \tau(Y_{n-3})        & \tau^2(Y_{n-4})           & \dots  & d\tau^{m-1}(Y_{n-1})\\
Y_{n-1} & \tau(Y_{n-2})        & \tau^2(Y_{n-3})           & \dots  &  \tau^{m-1}(Y_0) \end{array} \right]
\end{equation}
with
$Y_i\in \lambda(D_0\otimes_{F_0}F)$.
That means, $Y_i\in{\rm Mat}_{n}(K)$, and when the entries in $Y_i$ are restricted to elements in $L$,
$Y_i\in \lambda(D_0)$ (multiplication with $d$
in the upper right triangle of the matrix means simply scalar multiplication with $d$).

\begin{theorem} \label{main}
 (i)  $(L/F_0,\sigma,c)\otimes_{F_0} (F/F_0,\tau,d)\cong {\rm It}_R^m(D_0\otimes_{F_0}F,\tau,d)$.
\\ (ii)
Suppose that $D=(L/F_0,\sigma,c)\otimes_{F_0}F$ is a division algebra. Then
$$S_f\cong(L/F_0,\sigma,c)\otimes_{F_0} (F/F_0,\tau,d)$$
where $R=D[t;\widetilde{\tau}^{-1}]$ and
$f(t)=t^m-d$.
\end{theorem}

\begin{proof}
(i) The matrices in (\ref{eq:nbynass}) also represent left multiplication with an element in
the algebra ${\rm It}_R^m((K/F,\sigma,c),\tau,d)$, see (\ref{equ:main}). Thus the multiplications of both algebras are the same.
\\ (ii) If $D_0\otimes_{F_0}F$ is a division algebra then
 $S_f\cong {\rm It}_R^m((K/F,\sigma,c),\tau,d)$ with $R=(D_0\otimes_{F_0}F)[t;\widetilde{\tau}^{-1}]$ and
$f(t)=t^m-d$ by Theorem \ref{thm:skew}.
\end{proof}

\begin{corollary}
(i) The cyclic algebras
$$(K/L,\tau,d) \text{ and } (K/F,\sigma,c)$$
viewed as algebras over $F_0$,  are subalgebras of
$$(L/F_0,\sigma,c)\otimes_{F_0} (F/F_0,\tau,d)$$
 of dimension $m^2n$, resp. $n^2m$.
\\ (ii) The subalgebra $(K/L,\tau,d)$ is nonassociative and thus division if $m$ is prime
or, if $m$ is not prime, if $1,d,\dots,d^{m-1}$ are linearly independent over $L$.
\\ (iii)
If $m=st$ and $F_s={\rm Fix}(\tau^s)$  then
$$(L/F_0,\sigma,c)\otimes_{F_0} (F/F_s,\tau^s,d) $$
 is isomorphic to a subalgebra of
$$(L/F_0,\sigma,c)\otimes_{F_0} (F/F_0,\tau,d)={\rm It}_R^m(D_0\otimes_{F_0}F,\tau,d).$$
\end{corollary}

\begin{proof}
 (i) This is Lemma \ref{lem:lem5} and \cite{PS14}, Lemma 5 (which also holds if $D_0\otimes_{F_0}F$
is not division).
\\ (ii) This follows from (i), since $(F/F_0,\tau,d)$ is nonassociative if and only if $d\in F\setminus F_0$.
This means $d\in K\setminus L$. The same argument holds for nonassociative $(L/F_0,\sigma,c)$.
\\ (iii) This follows from \cite{S13}, Theorem 3.3.2, see also \cite{S12}.
\end{proof}

\subsection{Conditions on the tensor product to be a division algebra}
To see when the tensor product  of two associative algebras
is a division algebra we have the classical result by Jacobson \cite[Theorem 1.9.8]{J96}, see also Albert
\cite[Theorem 12, Ch. XI]{A}:

\begin{theorem} \label{associativetensor}
Let $(F/F_0,\tau,d)$ be a cyclic associative division algebra of prime degree $m$. Suppose that
 $D_0$
 is a central associative algebra over $F_0$ such that
$D=D_0\otimes_{F_0}F$ is a division algebra.
Then $D_0\otimes_{F_0} (F/F_0,\tau,d)$
is a division algebra if and only if
$$d\not=\widetilde{\tau}^{m-1}(z)\cdots\widetilde{\tau}(z)z$$
 for all $z\in D$.
\end{theorem}

Note that here
$$d\not=\widetilde{\tau}^{m-1}(z)\cdots\widetilde{\tau}(z)z
\text{ is equivalent to }d\not=z\widetilde{\tau}(z)\cdots \widetilde{\tau}^{m-1}(z)$$
since $d\in F_0$.
This classical result has the following immediate generalization to the nonassociative setting:

\begin{theorem} \label{thm:biquats}
 Let $(F/F_0,\tau,d)$  be a nonassociative cyclic algebra of degree $m$.
 Let $D_0=(L/F_0,\sigma,c)$ be an associative cyclic  algebra  over $F_0$ of degree $n$,
such that $D=D_0\otimes_{F_0}F=(K/F,\sigma,c)$ is a division algebra.
Let $L$ and $F$ be linearly disjoint over $F_0$.

Assume $m$ is prime and in case $m\not=2,3$, additionally that
 $F_0$ contains a primitive $m$th root of unity. Then
$$(L/F_0,\sigma,c)\otimes_{F_0} (F/F_0,\tau,d)$$
  is a division algebra if and only if
$$d\not=z\widetilde{\tau}(z)\cdots \widetilde{\tau}^{m-1}(z)
$$ for all $z\in D$.
\end{theorem}

\begin{proof}
This is Theorem \ref{main} together with  Theorem \ref{thm:betterpetit} (resp., \cite{P13}, Theorem 3.2 for
 $n=2$).
\end{proof}

More generally, we conclude:

\begin{theorem} \label{thm:classicalnew0}
Let $(F/F_0,\tau,d)$  be an associative or nonassociative cyclic algebra of degree $m$.
 Let $D_0=(L/F_0,\sigma,c)$ be an associative cyclic  algebra  over $F_0$ of degree $n$,
such that $D=D_0\otimes_{F_0}F=(K/F,\sigma,c)$ is a division algebra. Let $L$ and $F$ be linearly disjoint over $F_0$.

Then $$(L/F_0,\sigma,c)\otimes_{F_0} (F/F_0,\tau,d)$$
  is a division algebra if and only if
  $$f(t)=t^m-d\in D[t;\widetilde{\tau}^{-1}]$$
  is irreducible.
  \end{theorem}

This and Corollary  \ref{cor:necII} yields:

\begin{theorem} \label{thm:classicalnew}
Let $(F/F_0,\tau,d)$  be a nonassociative cyclic algebra of degree $m$.
 Let $D_0=(L/F_0,\sigma,c)$ be an associative cyclic  algebra  over $F_0$ of degree $n$,
such that $D=D_0\otimes_{F_0}F=(K/F,\sigma,c)$ is a division algebra.
Let $L$ and $F$ be linearly disjoint over $F_0$.
\\ (a)  Let $m=4$. Then $(L/F_0,\sigma,c)\otimes_{F_0} (F/F_0,\tau,d)$
  is a division algebra if and only if
$$d\not=z\widetilde{\tau}(z)\widetilde{\tau}^2(z)\widetilde{\tau}^3(z) $$
and
$$ \widetilde{\tau}^2(z_1)\widetilde{\tau}^3(z_1)z_1+\widetilde{\tau}^2(z_0)z_1+\widetilde{\tau}^2(z_1)
\widetilde{\tau}^3(z_0)\not=0 \text{ or } \widetilde{\tau}^2(z_0)z_0+\widetilde{\tau}^2(z_1)\widetilde{\tau}^3(z_0)z_0\not=d$$
 for all $z_0,z_1\in D$.
\\ (b) Suppose that
  $m$ is prime and in case $m\not=2,3$, additionally  that $F_0$ contains a primitive $m$th root of unity.
  \\ (i) If $\tau(d^n)\not=d^n$ then
 $$(L/F_0,\sigma,c)\otimes_{F_0} (F/F_0,\tau,d)$$
  is a division algebra.
 \\ (ii) If  $d^n\not\in N_{D/F_0}(D^\times)$, then
  $$(L/F_0,\sigma,c)\otimes_{F_0} (F/F_0,\tau,d)$$
 is a division algebra.
 \\ (iii) For all $d^n\not\in F_0$,
  $$(L/F_0,\sigma,c)\otimes_{F_0} (F/F_0,\tau,d)$$
 is a division algebra.
\end{theorem}

In special cases,
 Theorem \ref{thm:biquats} yields straightforward conditions for the tensor product to be a division algebra:

 \begin{theorem} \label{biquat}
 Let $F_0$ be of characteristic not 2.
 Let $(a,c)_{F_0}$ be a quaternion algebra over $F_0$ which is a division algebra over $F=F_0(\sqrt{b})$,
 and $(F_0(\sqrt{b})/F_0,\tau,d)$ a nonassociative quaternion algebra.
 Then the tensor product
$$(a,c)_{F_0}\otimes_{F_0} (F_0(\sqrt{b})/F_0,\tau,d)$$
 is a division algebra over $F_0$.
 \end{theorem}

\begin{proof}
Here, $K=F_0(\sqrt{a},\sqrt{b})$  with Galois group
 $G={\rm Gal}(K/F_0)=\{id, \sigma,\tau,\sigma\tau\}$, where
 $$\sigma(\sqrt{a})=-\sqrt{a},\quad \sigma(\sqrt{b})=\sigma(\sqrt{b}),$$
 $$\tau(\sqrt{a})=\sqrt{a},\quad \tau(\sqrt{b})=-\sqrt{b},$$
    $L=F_0(\sqrt{a})$ and $D=(a,c)_{F_0}\otimes F$.
For $z=z_0+iz_1+jz_2+kz_3\in D$, $z_i\in F_0(\sqrt{b})$, $i^2=a$, $j^2=c$, we get
$$z\widetilde{\tau}(z)=(z_0\tau(z_0)+az_1\tau(z_1)+cz_2\tau(z_2)-acz_3\tau(z_3))$$
$$+i(z_0\tau(z_1)+z_1\tau(z_0)-cz_2\tau(z_3)+cz_3\tau(z_2))$$
$$+j(z_0\tau(z_2)+z_2\tau(z_3)+az_1\tau(z_3)-az_3\tau(z_1))$$
$$+k(z_0\tau(z_3)+z_3\tau(z_2)+z_1\tau(z_2)-z_2\tau(z_1)).$$
Since $(F_0(\sqrt{b})/F_0,\tau,d)$ is nonassociative,  $d\in F_0(\sqrt{b})\setminus F_0$.
Hence if we assume that $d=z\widetilde{\tau}(z)$ for some $z\in D$ then
$$d= z_0\tau(z_0)+a\sigma(z_1)\tau(z_1)+c\sigma(z_2)\tau(z_2)-ac\sigma(z_3)\tau(z_3)$$
$$=N_{F/F_0}(z_0)+aN_{F/F_0}(z_1)+cN_{F/F_0}(z_2)-acN_{F/F_0}(z_3)\in F_0,$$
a contradiction. Thus, by Theorem \ref{thm:biquats},
 the tensor product
$$(a,c)_{F_0}\otimes_{F_0} (F_0(\sqrt{b})/F_0,\tau,d)$$
 is a division algebra.
\end{proof}

\begin{theorem}
 \label{3xquat}
 Let $F_0$ be of characteristic not 2, $F=F_0(\sqrt{b})$.
 Let $D_0=(L/F_0,\sigma,c)$ be a cyclic algebra over $F_0$ of degree $3$
 and $(F_0(\sqrt{b})/F_0,\tau,d)$ a nonassociative quaternion algebra. Let $d=d_0+\sqrt{b}d_1\in F\setminus F_0$
with $d_0,d_1\in F_0$.
\\ (i) If  $3d_0^2+bd_1^2\not=0$, then
$$(L/F_0,\sigma,c)\otimes_{F_0} (F_0(\sqrt{b})/F_0,\tau,d)$$
 is a division algebra over $F_0$.
 \\ (ii) Let $F_0=\mathbb{Q}$. If $b>0$, or if $b<0$ and $-\frac{b}{3}\not\in\mathbb{Q}^{\times 2}$
  then
$$(L/F_0,\sigma,c)\otimes_{F_0} (F_0(\sqrt{b})/F_0,\tau,d)$$
 is a division algebra over $F_0$.
 \end{theorem}

\begin{proof} Here, $F=F_0(\sqrt{b})$ and $K=F_0(\sqrt{b})$.
\\ (i) We know
$d^3=d_0^3+3bd_0d_1^2+\sqrt{b}d_1(3d_0^2+bd_1^2),$
so if we want that
$d^3\not=\widetilde{\tau}(d^3)$, this is equivalent to
 $3d_0^2+bd_1^2\not=0$. The assertion follows from Theorem \ref{thm:classicalnew} (b).
 \\ (ii) is a direct consequence from (i): for $F_0=\mathbb{Q}$, $3d_0^2+bd_1^2>0$ for all $b>0$. For $b<0$,
 the assertion is true since
 $3d_0^2+bd_1^2=0$ if and only if
 $(\frac{d_0}{d_1})^2=-\frac{b}{3}$.

\end{proof}

\begin{remark}
 For $A=(L/F_0,\sigma,c)\otimes_{F_0} (F/F_0,\tau,d)$, the map  $M_{A}(x)=\det(L_x)=\det (\lambda(M(x)))$
 can be seen as a generalization
 of the norm of an associative central simple algebra, since $M_{A}=N_{A/F}$ if both cyclic algebras in the
 tensor product $A$ are associative.

 For all $X=\lambda(M(x)) =\lambda(x)\in
\lambda(A)\subset {\rm Mat}_{nm}(K),$  $D=D_0\otimes_{F_0} F$, we have
 ${\rm det}\, X\in F$
 (cf. \cite{PS13} or
 \cite[Corollary 2]{P13} for $m=2$). Thus  $M_{A}:A\to F$.
 We also have
 $$M_A(x)=N_{D/F}(x)\tau(N_{D/F}(x))\cdots\tau(N_{D/F}(x))=N_{F/F_0}(N_{D/F}(x))$$
 for all $x\in (K/F,\sigma,c)$ (which is easy to see from applying the determinant to the matrix of $L_x$ in
  Equation (4) for some $x\in D$).
\end{remark}

We conclude with the observation that the generalization of Albert and Jacobson's condition is a necessary condition for $d$ in the general case:

\begin{proposition}
Let $D_0=(L/F_0,\sigma,c)$ be a an associative cyclic  algebra of degree $n$ over $F_0$,
such that $D=D_0\otimes_{F_0}F$ is a division algebra.
 If $D_0\otimes_{F_0} (F/F_0,\tau,d)$ is a division algebra then
 $$d\not=z\widetilde{\tau}(z)\cdots \widetilde{\tau}^{m-1}(z)$$
 for all $z\in D$.
\end{proposition}

\section{Applications to space-time block coding}

Space-time block codes  are used for reliable high rate transmission
over wireless digital channels with multiple antennas transmitting and receiving the data.
 For instance, four transmit and two receive antennas can be used in
digital video broadcasting for portable TV devices, or for transmitting data to mobile phones.

A \emph{space-time block code} (STBC) is a set $\mathcal{C}$ of complex $n\times m$ matrices (the \emph{codebook}),
 that satisfies a number of properties which determine how well the code performs.
Each column is transmitted from $n$ transmit antennas simultaneously, after $n$ columns are transmitted,
the receive antennas process the data and try to recover them.

We consider \emph{linear} codes: if $X,X'\in \mathcal{C}$ then $X\pm X'\in \mathcal{C}$.
There are some basic design criteria which are desirable for good performance of a linear code:
 The code should be {\it fully diverse}, i.e. ${\rm det}X\not=0$ for all $0\not=X\in \mathcal{C}$;
and for all $X\in \mathcal{C}$, $|{\rm det}X|^2$ should be as large as possible to minimize
 the pairwise error probability.
The linear code has {\it non-vanishing determinant} (NVD) if
${\rm min}_{\{X\in\,\mathcal{C}\}}|{\rm det}X|^2$ is bounded away from 0.

Central simple (associative) division algebras, in particular cyclic division algebras $A=(K/F,\sigma, c)$ of degree $m$ with $K$
an imaginary number field, have been
 highly successfully used to systematically build STBCs, using the fact that their left regular representation
 $$\lambda: A \hookrightarrow {\rm End}_K(A)\subset {\rm Mat}_m(K),$$
 $$ x \mapsto L_x\mapsto X$$
 over a maximal subfield $K$
 yields a fully diverse linear STBC $\mathcal{C}=\lambda(A)$.

 The idea to take the matrix representing  left
 multiplication (over a suitable subfield of the algebra) to obtain a fully diverse STBC can be extended to nonassociative division algebras,
 provided they have a large enough nucleus, as first observed in \cite{PU11}.

The previous results establish a general framework and some new simplified conditions for constructing fully diverse STBCs
using the left multiplication in
tensor products of any associative cyclic  algebra  over $F_0$ of degree $n$ and a
nonassociative cyclic algebra of degree $m$. We thus have the potential to systematically construct
STBCs for given numbers of transmit/receive antennas which are not just fully diverse but fast-decodable, provided we start with a suitable
cyclic algebra (whose left regular representation yields a fast-decodable code), as discussed in
\cite[Section 6.1]{PS14}.

 We  use the $nm^2$ degrees of freedom of the
channel to transmit $nm^2$ complex information symbols per \emph{codeword} $X\in \mathcal{C}$.
If $mn$ channels are
used, our space-time block code $\mathcal{C}$ consisting of $mn\times mn$ matrices $X$ of the form
 (\ref{eq:nbynass}) with entries in $K$ therefore
 has a rate of $m$ complex symbols per channel use, which is maximal for $m$ receive antennas.
 If the associated tensor product algebra is division (which can be checked
 using Theorem \ref{main} and the resulting easier conditions for special cases), the code is fully diverse.
This generalizes the set-up discussed in \cite[Section V.A]{MO13} which only covers $n=3$, $m=2$.

 Theorem  \ref{main}
implies also that the fast-decodable iterated codes for 6 transmit and 3 receive antennas
constructed seemingly ad hoc in  \cite[Section V.A]{MO13} consist of the $6\times6$-matrices representing left
multiplication in the tensor product of a degree three cyclic division algebra $D$ and a nonassociative quaternion algebra.

The division algebra
$$A=(-1,-1)_\mathbb{Q}\otimes_\mathbb{Q} (\mathbb{Q}(\theta)/\mathbb{Q},\tau,\theta )$$
 is behind the fast-decodable fully diverse code designed in \cite[Section 6.2]{PS14}, where
 $\omega=\frac{-1+\sqrt{3}i}{2}$ is a primitive third root of unity,
 $\zeta_7$ is a primitive $7^{th}$ root of
unity, $\theta = \zeta_7 + \zeta_7^{-1} = 2
\cos(\frac{2 \pi}{7})$,  and  $\mathbb{Q}(\omega, \theta)/\mathbb{Q}(\omega)$ is a cubic cyclic field extension
whose Galois group is generated by the automorphism $\tau: \zeta_7 +\zeta_7^{-1} \mapsto \zeta_7^2 + \zeta_7^{-2}$.
For $m=3$ and $n=2$ we can now systematically build other fast-decodable fully diverse iterated codes of maximal rate
 for 6 transmit and 3 receive antennas out of tensor products if desired.

Using Theorem \ref{biquat}, we can construct a new family of fast-decodable $4\times 2$-codes starting with the Silver code.
These codes have the same  decoding complexity as the SR-code \cite[IV.B]{R13}
 and as the  code presented in \cite{LHHC}, which is $O(M^{4.5})$. Three
 other state-of the-art `ad hoc'  constructions of $4\times 2$-codes of the same decoding complexity are given in
  \cite[Table I]{LHHC}.
  To our knowledge  there are no fast-decodable fully diverse $4\times 2$ codes known
  for 4x2 multiple-input multiple-output transmission with a better decoding complexity than $O(M^{4.5})$:

\begin{example}
Let $F=\mathbb{Q}(\sqrt{-7})$, $K=\mathbb{Q}(i,\sqrt{-7})$ and
$L=\mathbb{Q}(i)$. $D=(-1,-1)_F$ is the quaternion division algebra used in the Silver code,
and for $a,b\in\mathbb{Q}(\sqrt{-7})$, $\sigma(a+ib)=a-ib$.
 The entries of the matrices of the Silver code are elements of the order
$\mathcal{O}_K\oplus j\mathcal{O}_K$ in $D=(-1,-1)_{\mathbb{Q}(\sqrt{-7})}$, where
$\mathcal{O}_K=\mathbb{Z}[i]\oplus\mathbb{Z}[i](\frac{1+\sqrt{-7}}{2})$ is the ring of integers of $K$.

By Theorem \ref{biquat},  $A=(-1,-1)_\mathbb{Q}\otimes_\mathbb{Q} (\mathbb{Q}(\sqrt{-7})/\mathbb{Q},\tau,d)$
 is a division
algebra for all choices of $d\in \mathbb{Q}(\sqrt{-7})\setminus\mathbb{Q}$. Therefore the code $\mathcal{C}$ given by
the matrix
\[
 \left [\begin {array}{cccc}
x_0 &  -\sigma(x_1)  &  d\tau(y_0)  & -d\tau(\sigma(y_1))\\
x_1 & \sigma(x_0)    & d\tau(y_1)   & d \tau(\sigma(y_0))\\
y_0 & - \sigma(y_1)  &  \tau(x_0)  &  -\tau(\sigma( x_1))\\
y_1 & \sigma(y_0)    &  \tau(x_1)  &  \tau(\sigma(x_0))\\
\end {array}\right ],
 \]
representing left multiplication in $A$, $x_i,y_i\in \mathbb{Q}(\sqrt{-7})$,
 is fully diverse and has NVD since $\det X\in \mathbb{Q}(\sqrt{-7})$ \cite{PS13}, for all choices of
$d\in \mathbb{Q}(\sqrt{-7})\setminus\mathbb{Q}$. Therefore its minimum determinant is at least 1.
This rough estimate
 puts our code family already among the top three places in \cite[Table I]{LHHC}, along with
 the code proposed in \cite{LHHC} and the SR-code, which compares the
fast-decodable codes with best minimum determinant and decoding complexity. This is because
$\mathcal{C}$   has decoding complexity at most
$$\mathcal{O}(M^{2m^2-3m/2})=\mathcal{O}(M^{5})$$
no matter the choice of $d\in \mathbb{Q}(\sqrt{-7})\setminus\mathbb{Q}$,
if the matrix entries take values from $M$-QAM $\subset\mathbb{Z}[i]$, thus is fast-decodable of maximal rate 2 by
\cite[Section 6.1]{PS14}.
In particular, it has optimal diversity-multiplexing gain trade-off (DMT).
The usual hard-limiting as done for the SR-code then lowers its decoding complexity to $\mathcal{O}(M^{4.5})$.

 As a further comparison, the iterated Silver code in \cite[Section IV.A]{MO13}, which does not arise from a tensor product (since
 there $\sigma=\tau$), initially has decoding complexity at most
$\mathcal{O}(M^{13})$ which is then further reduced to $\mathcal{O}(M^{10})$ by scaling.
\end{example}


\end{document}